\documentclass[runningheads]{llncs}
\usepackage{graphicx}
\usepackage{longtable}
\usepackage{wrapfig}
\usepackage{ucs}
\usepackage{dsfont}
\usepackage{mathrsfs}
\usepackage{mathtools}
\usepackage{amsmath}
\usepackage{amsfonts}
\usepackage{amssymb}
\usepackage{soul}
\usepackage[makeroom]{cancel}
\usepackage{bbm}
\usepackage[english]{babel}
\usepackage{ucs}
\usepackage[colorlinks=true,linkcolor=blue,citecolor=blue,urlcolor=orange]{hyperref}
\usepackage{longtable}
\usepackage{wasysym}
\usepackage{verbatim}
\usepackage{multirow}
\usepackage{bussproofs}
\usepackage{setspace}
\usepackage{graphicx}
\usepackage{stmaryrd}
\usepackage{forest}
\forestset{smullyan tableaux/.style={for tree={math content},where n children=1{!1.before computing xy={l=\baselineskip},!1.no edge}{},closed/.style={label=below:$\times$},},}
\usepackage{longtable}
\usepackage[all]{xy}
\usepackage{wrapfig}
\usepackage{xcolor}
\usepackage[colorinlistoftodos,prependcaption,textsize=small]{todonotes}
\usepackage[most]{tcolorbox}
\usepackage{enumitem}
\usepackage{tikz}

\newcommand{\commentSabine}[1]{}

\newcommand{\FDE}{\textsf{FDE}}

\newcommand{\BD}{\textsf{BD}}

\newcommand{\coimplies}{\Yleft}
\newcommand{\Gsquare}{\mathsf{G}^2}
\newcommand{\KGsquare}{\mathbf{K}\mathsf{G}^2}
\newcommand{\KG}{\mathfrak{GK}}
\newcommand{\KbiG}{\mathbf{K}\mathsf{biG}}
\newcommand{\fbKGsquare}{\mathbf{K}\mathsf{G}^2_{\mathsf{fb}}}

\newcommand{\fbKbiG}{\mathbf{K}\mathsf{biG}_{\mathsf{fb}}}
\newcommand{\bimodalLsquare}{\mathsf{bi}\mathcal{L}^\neg_{\Box,\lozenge}}
\newcommand{\bimodalL}{\mathsf{bi}\mathcal{L}_{\Box,\lozenge}}

\newtheorem{convention}{Convention}

\usepackage{comment}

\begin{document}

\setlength{\jot}{0pt} 
\setlength{\abovedisplayskip}{2pt}
\setlength{\belowdisplayskip}{2pt}
\setlength{\abovedisplayshortskip}{1pt}
\setlength{\belowdisplayshortskip}{1pt}

\title{Paraconsistent G\"{o}del modal logic\thanks{The research of Marta B\'ilkov\'a was supported by the grant 22-01137S of the Czech Science Foundation. The research of Sabine Frittella and Daniil Kozhemiachenko was funded by the grant ANR JCJC 2019, project PRELAP (ANR-19-CE48-0006). This research is part of the MOSAIC project financed by the European Union's Marie Sk\l{}odowska-Curie grant No.~101007627.}}
%
\author{Marta B\'ilkov\'a\inst{1}\orcidID{0000-0002-3490-2083} \and Sabine Frittella\inst{2}\orcidID{0000-0003-4736-8614}\and Daniil Kozhemiachenko\inst{2}\orcidID{0000-0002-1533-8034}}
\authorrunning{B\'ilkov\'a et al.}
\institute{The Czech Academy of Sciences, Institute of Computer Science, Prague\\
\email{bilkova@cs.cas.cz}
\and
INSA Centre Val de Loire, Univ.\ Orl\'{e}ans, LIFO EA 4022, France\\
\email{sabine.frittella@insa-cvl.fr, daniil.kozhemiachenko@insa-cvl.fr}}
\maketitle              
\begin{abstract}
We introduce a~paraconsistent modal logic $\KGsquare$, based on G\"{o}del logic with coimplication (bi-G\"{o}del logic) expanded with a De Morgan negation $\neg$. We use the logic to formalise reasoning with graded, incomplete and inconsistent information. Semantics of $\KGsquare$ is two-di\-men\-si\-o\-nal: we interpret $\KGsquare$ on crisp frames with two valuations $v_1$ and $v_2$, connected via $\neg$, that assign to each formula two values from the real-valued interval $[0,1]$. The first (resp., second) valuation encodes the positive (resp., negative) information the state gives to a~statement. We obtain that $\KGsquare$ is strictly more expressive than the classical modal logic $\mathbf{K}$ by proving that finitely branching frames are definable and by establishing a faithful embedding of $\mathbf{K}$ into $\KGsquare$. We also construct a~constraint tableau calculus for $\KGsquare$ over finitely branching frames, establish its decidability and provide a~complexity evaluation.
\end{abstract}
\section{Introduction}\label{sec:introduction}
People believe in many things. Sometimes, they even have contradictory beliefs. Sometimes, they believe in one statement more than in the other. However, if a person has contradictory beliefs, they are not bound to believe in anything. Likewise, believing in $\phi$ \emph{strictly more than} in $\chi$ makes one believe in $\phi$ \emph{completely}. These properties of beliefs are natural, and yet hardly expressible in the classical modal logic. In this paper, we present a two-dimensional modal logic based on G\"{o}del logic that can formalise beliefs taking these traits into account.

\vspace{.8em}

\noindent
\textbf{\textit{Two-dimensional treatment of uncertainty.}}
Belnap-Dunn four-valued logic ($\BD$, or First Degree Entailment --- $\FDE$)~\cite{Belnap19,Dunn1976,OmoriWansing2017} can be used to formalise reasoning with both incomplete and inconsistent information. In $\BD$, formulas are evaluated on the De Morgan algebra $\mathbf{4}$ (fig.~\ref{fig:square:with:filter}, left) where the four values $\{t,f,b,n\}$ encode the information available about the formula: true, false, both true and false, neither true nor false. $b$ and $n$ thus represent inconsistent and incomplete information, respectively. It is important to note that the values represent the available information about the statement, not its intrinsic truth or falsity. Furthermore, this approach essentially treats \emph{evidence for} a statement (its positive support) as being independent of \emph{evidence against} it (negative support) which allows to differentiate between ‘absence of evidence’ and the ‘evidence of absence’. The $\BD$ negation $\neg$ then swaps positive and negative supports.
\begin{figure}
\centering
\resizebox{0.26\textwidth}{!}{
\begin{tikzpicture}[>=stealth,relative]
\node (U1) at (0,-1) { $f$ };
\node (U2) at (-1,0) { $n$ };
\node (U3) at (1,0) { $b$ };
\node (U4) at (0,1) { $t$ };
\path[-,draw] (U1) to (U2);
\path[-,draw] (U1) to (U3);
\path[-,draw] (U2) to (U4);
\path[-,draw] (U3) to (U4);
\end{tikzpicture}
}
\quad
\resizebox{0.28\textwidth}{!}{
\begin{tikzpicture}[>=stealth,relative]
\node (U1) at (0,-2) {$(0,1)$};
\node (U2) at (-2,0) {$(0,0)$};
\node (U3) at (2,0) {$(1,1)$};
\node (U4) at (0,2) {$(1,0)$};
\path[-,draw] (U1) to (U2);
\path[-,draw] (U1) to (U3);
\path[-,draw] (U2) to (U4);
\path[-,draw] (U3) to (U4);
\end{tikzpicture}
}
\caption{$\mathbf{4}$ (left) and its continuous  extension $[0,1]^\Join$ (right). $(x,y)\leq_{[0,1]^\Join}(x',y')$ iff $x\leq x'$ and $y\geq y'$.}
\label{fig:square:with:filter}
\end{figure}
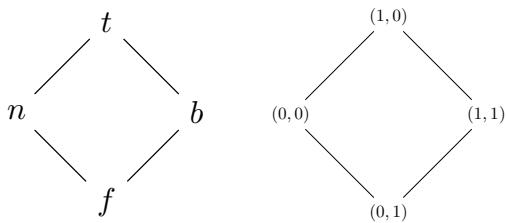

The information regarding a statement, however, might itself be not crisp --- after all, our sources are not always completely reliable. Thus, 
to capture the uncertainty, we 
extend $\mathbf{4}$ to the lattice $[0,1]^\Join$ (fig.~\ref{fig:square:with:filter}, right). $[0,1]^\Join$ is a twist  product (cf,~\cite{Vakarelov1977} for definitions) of $[0,1]$ with itself: the order on the second coordinate is reversed w.r.t. the order on the first coordinate. This captures the intuition behind the usual ‘truth’ (upwards) order: an agent is more certain in $\chi$ than in $\phi$ when the evidence for $\chi$ is stronger than the evidence for $\phi$ while the evidence against $\chi$ is weaker than the evidence against $\phi$.

Note that $[0,1]^\Join$ is a bilattice whose left-to-right order can be interpreted as the information order. This links the logics we consider to bilattice logics applied to reasoning in AI in~\cite{Ginsberg1988} and then studied further in~\cite{Riveccio2010PhD,JansanaRamon2012}.


\vspace{.8em}

\noindent
\textbf{\textit{Comparing beliefs.}}
Uncertainty is manifested not only in the non-crisp character of the information. An agent might often lack the capacity to establish the concrete numerical value that represents their certainty in a given statement. Indeed, ‘I am 43\% certain that the wallet is Paula's’ does not sound natural. On the other hand, it is reasonable to assume that the agents' beliefs can be compared in most contexts: neither ‘I~am more confident that the wallet is Paula's than that the wallet is Quentin's’, nor ‘Alice is more certain than Britney that Claire loves pistachio ice cream’ require us to give a concrete numerical representation to the (un)certainty.

These considerations lead us to choosing the two-dimensional relative of the G\"{o}del logic dubbed $\Gsquare$ as the propositional fragment of our logic. $\Gsquare$ was introduced in~\cite{BilkovaFrittellaKozhemiachenko2021} and is, in fact, an extension of Moisil's logic\footnote{This logic was introduced several times: by Wansing~\cite{Wansing2008} as $\mathsf{I}_4\mathsf{C}_4$ and then by Leitgeb~\cite{Leitgeb2019} as HYPE. Cf.~\cite{OdintsovWansing2021} for a recent and more detailed discussion.} from~\cite{Moisil1942} with the prelinearity axiom $(p\rightarrow q)\vee(q\rightarrow p)$. As in the original G\"{o}del logic $\mathsf{G}$, the validity of a formula in $\Gsquare$ depends not on the values of its constituent variables but on the relative order between them. In this sense, $\mathsf{G}$ is a logic of comparative truth. Thus, as we treat positive and negative supports of a given statement independently, $\Gsquare$ is a logic of comparative truth and falsity. Note that while the values of two statements may not be comparable (say, $p$ is evaluated as $(0.5,0.3)$ and $q$ as $(0,0)$), the coordinates of the values always are. We will see in section~\ref{sec:language}, how we can formalise statements comparing agents' beliefs.

The sources available to the agents as well as the references between these sources can be represented as states in a Kripke model and its accessibility relation, respectively. It is important to mention that we account for the possibility that a~source can give us contradictory information regarding some statement. Still, we want our reasoning with such information to be non-trivial. This is reflected by the fact that $(p\wedge\neg p)\rightarrow q$ is not valid in $\Gsquare$. Thus, the logic (treated as a set of valid formulas) lacks the explosion principle. In this sense, we call $\Gsquare$ and its modal expansions ‘paraconsistent’. This links our approach to other paraconsistent fuzzy logics such as the ones discussed in~\cite{ErtolaEstevaFlaminioGodoNoguera2015}.

To reason with the information provided by the sources, we introduce two interdefinable modalities --- $\Box$ and $\lozenge$ --- interpreted as infima and suprema w.r.t. the upwards order on $[0,1]^\Join$. We mostly assume (unless stated otherwise) that accessibility relations in models are crisp. Intuitively, it means that the sources are either accessible or not (and, likewise, either refer to the other ones, or not).


\vspace{.8em}

\noindent
\textbf{\textit{Broader context.}}
This paper is a part of the project introduced in~\cite{BilkovaFrittellaMajerNazari2020} and carried on in~\cite{BilkovaFrittellaKozhemiachenko2021} aiming to de\-ve\-lop a modular logical framework for reasoning based on uncertain, incomplete and inconsistent information. We model agents who build their epistemic attitudes (like beliefs) based on information aggregated from multiple sources. $\Box$ and $\lozenge$ can be then viewed as two simple aggregation strategies: a pessimistic one (the infimum of positive support and the supremum of the negative support), and an optimistic one (the dual strategy), respectively. They can be defined via one another using $\neg$ in the expected manner: $\Box\phi$ stands for $\neg\lozenge\neg\phi$ and $\lozenge\phi$ for $\neg\Box\neg\phi$. In this paper, in contrast to~\cite{CintulaNoguera2014} and~\cite{BilkovaFrittellaMajerNazari2020}, we do allow for modalities to nest.

The other part of our motivation comes from the work on modal G\"{o}del logic ($\KG$ --- in the notation of~\cite{RodriguezVidal2021}) equipped with relational semantics~\cite{CaicedoRodriguez2010,CaicedoRodriguez2015,RodriguezVidal2021}. There, the authors develop proof and model theory of modal expansions of $\mathsf{G}$ interpreted over frames with both crisp and fuzzy accessibility relations. In particular, it was shown that the $\Box$-fragment\footnote{Note that $\Box$ and $\lozenge$ are not interdefinable in $\KG$ --- cf.~\cite[Lemma~6.1]{RodriguezVidal2021} for details.} of $\KG$ lacks the finite model property (FMP) w.r.t.\ fuzzy frames while the $\lozenge$-fragment has FMP\footnote{There is, however, a~semantics in~\cite{CaicedoMetcalfeRodriguezRogger2013} w.r.t.\ which bi-modal $\KG$ has FMP.} only w.r.t. fuzzy (but not crisp) frames. Furthermore, both $\Box$ and $\lozenge$ fragments of~$\KG$ are \textsf{PSPACE}-complete~\cite{MetcalfeOlivetti2009,MetcalfeOlivetti2011}.

Description  G\"{o}del logics, a notational version of modal logics, have found their use the field of knowledge representation~\cite{BobilloDelgadoGomez-RamiroStraccia2009,BobilloDelgadoGomez-RamiroStraccia2012,BorgwardtDistelPenaloza2014}, in particular, in the representation of vague or uncertain data which is not possible in the classical ontologies. In this respect, our paper provides a further extension of representable data types as we model not only vague reasoning but also non-trivial reasoning with inconsistent information.

In the present paper, we are expanding the language with the G\"{o}del coimplication $\coimplies$ to allow for the formalisation of statements expressing that an agent is \emph{strictly more confident} in one statement than in another one (cf.~section~\ref{sec:language} for the details). Furthermore, the presence of $\neg$ will allow us to simplify the frame definability. Still, we will show that our logic is a conservative extension of $\KG^c$ --- the modal G\"{o}del logic of crisp frames from~\cite{RodriguezVidal2021} in the language with both $\Box$ and $\lozenge$.

\vspace{.8em}

\noindent
\textbf{\textit{Logics.}}
We are discussing many logics obtained from the propositional G\"{o}del logic $\mathsf{G}$. Our main interest is in the logic we denote $\KGsquare$. It can be produced from $\mathsf{G}$ in several ways: (1) adding De Morgan negation $\neg$ to obtain $\Gsquare$ (in which case $\phi\coimplies\phi'$ can be defined as $\neg(\neg\phi'\rightarrow\neg\phi)$) and then further expanding the language with $\Box$ or $\lozenge$; (2) adding $\coimplies$ or $\Delta$ (Baaz' delta) to $\mathsf{G}$, then both $\Box$ and $\lozenge$ thus acquiring $\KbiG$\footnote{To the best of our knowledge, the only work on bi-G\"{o}del (symmetric G\"{o}del) modal logic is~\cite{GrigoliaKiseliovaOdisharia2016}. There, the authors propose an expansion of $\mathsf{biG}$ with $\Box$ and $\lozenge$ equipped with proof-theoretic interpretation and provide its algebraic semantics.} (modal bi-G\"{o}del logic) which is further enriched with $\neg$. These and other relations are given on fig.~\ref{fig:logics}.
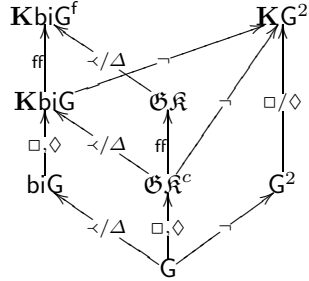
\begin{figure}
\centering
\[\xymatrix{
{\KbiG}^\mathsf{f}&&\KGsquare\\
\KbiG\ar[urr]|{\neg}\ar[u]^{\mathsf{ff}}&\KG\ar[ul]|{\coimplies/\Delta}&\\
\mathsf{biG}\ar[u]|{\Box,\lozenge}&\KG^c\ar[uur]|{\neg}\ar[u]^{\mathsf{ff}}\ar[ul]|{\coimplies/\Delta}&\Gsquare\ar[uu]|{\Box/\lozenge}\\
&\mathsf{G}\ar[ul]|{\coimplies/\Delta}\ar[u]|{\Box,\lozenge}\ar[ur]|{\neg}
}\]
\caption{Logics in the article. $\mathsf{ff}$ stands for ‘permitting fuzzy frames’. Subscripts on arrows denote language expansions. $/$ stands for ‘or’ and comma for ‘and’.}
\label{fig:logics}
\end{figure}

\noindent
\textbf{\textit{Plan of the paper.}}
The remainder of the paper is structured as follows. In section~\ref{sec:language}, we define bi-G\"{o}del algebras and use them to present $\KbiG$ (on both fuzzy and crisp frames) and then $\KGsquare$ (on crisp frames), show how to formalise statements where beliefs of agents are compared, and prove some semantical properties. In section~\ref{sec:modeltheory}, we show that $\lozenge$ fragment of $\KbiG^\mathsf{f}$ ($\KbiG$ on fuzzy frames) lacks finite model property. We then present a finitely branching fragment of $\KGsquare$ ($\fbKGsquare$) and argue for its use in representation of agents' beliefs. In section~\ref{sec:tableaux}, we design a~constraint tableaux calculus for $\fbKGsquare$ which we use to obtain the complexity results. Finally, in section~\ref{sec:conclusion} we discuss further lines of research.
\section{Language and semantics}\label{sec:language}
In this section, we present semantics for $\KbiG$ (modal bi-G\"{o}del logic) over both fuzzy and crisp frames and the one for $\KGsquare$ over crisp frames. 
Let $\mathsf{Var}$ be a~countable set of propositional variables. The language $\bimodalLsquare$ is defined via the following grammar.
\begin{align*}
\phi&\coloneqq p\in\mathsf{Var}\mid\neg\phi\mid(\phi\wedge\phi)\mid(\phi\vee\phi)\mid(\phi\rightarrow\phi)\mid(\phi\coimplies\phi)\mid\Box\phi\mid\lozenge\phi
\end{align*}
Two constants, $\mathbf{0}$ and $\mathbf{1}$, can be introduced in the traditional fashion: $\mathbf{0}\coloneqq p\coimplies p$, $\mathbf{1}\coloneqq p\rightarrow p$. Likewise, the G\"{o}del negation can be also defined as expected: ${\sim}\phi\coloneqq\phi\rightarrow\mathbf{0}$. The $\neg$-less fragment of $\bimodalLsquare$ is denoted with $\bimodalL$.

To facilitate the presentation, we introduce bi-G\"{o}del algebras.
\begin{definition}
The bi-G\"{o}del algebra $[0,1]_{\mathsf{G}}=([0,1],0,1,\wedge_\mathsf{G},\vee_\mathsf{G},\rightarrow_{\mathsf{G}},\coimplies_\mathsf{G})$ is defined as follows: for all $a,b\in[0,1]$, the standard operations are given by $a\wedge_\mathsf{G}b\coloneqq\min(a,b)$, $a\vee_\mathsf{G}b\coloneqq\max(a,b)$,
\begin{equation*}
a\rightarrow_{G} b=
\begin{cases}
1, \ \text{if}\  a\leq b\\
b \ \ \text{else,}
\end{cases}
\ \ \ \qquad
b\Yleft_{G} a=
\begin{cases}
0, \ \text{if}\  b\leq a\\
b \ \ \text{else.}
\end{cases}
\end{equation*}
\end{definition}
\begin{definition}\label{def:frames}
\begin{itemize}
\item[]
\item A \emph{fuzzy frame} is a tuple $\mathfrak{F}=\langle W,R\rangle$ with $W\neq\varnothing$ and $R:W\times W\rightarrow[0,1]$.
\item A \emph{crisp frame} is a tuple $\mathfrak{F}=\langle W,R\rangle$ with $W\neq\varnothing$ and $R\subseteq W\times W$.
\end{itemize}
\end{definition}
\begin{definition}[$\KbiG$ models]\label{def:KbiGsemantics}
A \emph{$\KbiG$ model} is a tuple $\mathfrak{M}=\langle W,R,v\rangle$ with $\langle W,R\rangle$ being a (crisp or fuzzy) frame, and $v:\mathsf{Var}\times W\rightarrow[0,1]$. $v$ (a valuation) is extended on complex $\bimodalL$ formulas as follows:
\begin{align*}
v(\phi\circ\phi',w)&=v(\phi,w)\circ_\mathsf{G}v(\phi',w).\tag{$\circ\in\{\wedge,\vee,\rightarrow,\coimplies\}$}
\end{align*}
The interpretation of modal formulas on \emph{fuzzy} frames is as follows:
\begin{align*}
v(\Box\phi,w)&=\inf\limits_{w'\in W}\{wRw'\rightarrow_\mathsf{G}v(\phi,w')\},&v(\lozenge\phi,w)&=\sup\limits_{w'\in W}\{wRw'\wedge_\mathsf{G}v(\phi,w')\}.
\end{align*}
On \emph{crisp} frames, the interpretation is simpler (here, $\inf(\varnothing)\!=\!1$ and $\sup(\varnothing)\!=\!0$):
\begin{align*}
v(\Box\phi,w)&=\inf\{v(\phi,w'):wRw'\},&v(\lozenge\phi,w)&=\sup\{v(\phi,w'):wRw'\}.
\end{align*}
We say that $\phi\in\bimodalL$ is \emph{$\KbiG$ valid on frame $\mathfrak{F}$} (denote, $\mathfrak{F}\models_{\KbiG}\phi$) iff for any $w\in\mathfrak{F}$, it holds that $v(\phi,w)=1$ for any model $\mathfrak{M}$ on $\mathfrak{F}$.
\end{definition}
Note that the definitions of validity in $\KG^c$ and $\KG$ coincide with those in $\KbiG$ and $\KbiG^\mathsf{f}$ if we consider the $\coimplies$-free fragment of $\bimodalL$.

As we have already mentioned, on \emph{crisp} frames, the accessibility relation can be understood as availability of (trusted or reliable) sources. In \emph{fuzzy} frames, it can be thought of as the degree of trust one has in a source. Then, $\lozenge\phi$ represents the search for evidence from trusted sources that supports $\phi$: $v(\lozenge\phi,t)>0$ iff there is $t'$ s.t.\ $tRt'>0$ and $v(\phi,t')>0$, i.e., there must be a source $t'$ to which $t$ has positive degree of trust and that has at least some certainty in $\phi$. On the other hand, if no source is trusted by $t$ (i.e., $tRu=0$ for all $u$), then $v(\lozenge\phi,t)=0$. Likewise, $\Box\chi$ can be construed as the search of evidence against $\chi$ given by trusted sources: $v(\Box\chi,t)<1$ iff there is a~source $t'$ that gives to $\chi$ less certainty than $t$ gives trust to $t'$. In other words, if $t$ trusts no sources, or if all sources have at least as high confidence in $\chi$ as $t$ has in them, then $t$ fails to find a trustworthy enough counterexample.
\begin{definition}[$\KGsquare$ models]\label{def:KG2semantics}
A \emph{$\KGsquare$ model} is a tuple $\mathfrak{M}=\langle W,R,v_1,v_2\rangle$ with $\langle W,R\rangle$ being a \emph{crisp} frame, and $v_1,v_2:\mathsf{Var}\times W\rightarrow[0,1]$. The valuations which we interpret as support of truth and support of falsity, respectively, are extended on complex formulas as expected.
\begin{longtable}{rclrcl}
$v_1(\neg\phi,w)$&$=$&$v_2(\phi,w)$&$v_2(\neg\phi,w)$&$=$&$v_1(\phi,w)$\\
$v_1(\phi\wedge\phi',w)$&$=$&$v_1(\phi,w)\wedge_\mathsf{G}v_1(\phi',w)$&$v_2(\phi\wedge\phi',w)$&$=$&$v_2(\phi,w)\vee_\mathsf{G}v_2(\phi',w)$\\
$v_1(\phi\vee\phi',w)$&$=$&$v_1(\phi,w)\vee_\mathsf{G}v_1(\phi',w)$&$v_2(\phi\vee\phi',w)$&$=$&$v_2(\phi,w)\wedge_\mathsf{G}v_2(\phi',w)$\\
$v_1(\phi\rightarrow\phi',w)$&$=$&$v_1(\phi,w)\!\rightarrow_\mathsf{G}\!v_1(\phi',w)$&$v_2(\phi\rightarrow\phi',w)$&$=$&$v_2(\phi',w)\coimplies_\mathsf{G}v_2(\phi,w)$\\
$v_1(\phi\coimplies\phi',w)$&$=$&$v_1(\phi,w)\coimplies_\mathsf{G}v_1(\phi',w)$&$v_2(\phi\coimplies\phi',w)$&$=$&$v_2(\phi',w)\!\rightarrow_\mathsf{G}\!v_2(\phi,w)$\\
$v_1(\Box\phi,w)$&$=$&$\inf\{v_1(\phi,w'):wRw'\}$&$v_2(\Box\phi,w)$&$=$&$\sup\{v_2(\phi,w'):wRw'\}$\\
$v_1(\lozenge\phi,w)$&$=$&$\sup\{v_1(\phi,w'):wRw'\}$&$v_2(\lozenge\phi,w)$&$=$&$\inf\{v_2(\phi,w'):wRw'\}$
\end{longtable}
%
We say that $\phi\in\bimodalLsquare$ is \emph{$\KGsquare$ valid on frame $\mathfrak{F}$} ($\mathfrak{F}\models_{\KGsquare}\phi$) iff for any $w\in\mathfrak{F}$, it holds that $v_1(\phi,w)=1$ and $v_2(\phi,w)=0$ for any model $\mathfrak{M}$ on $\mathfrak{F}$.
\end{definition}
\begin{convention}
In what follows, we will denote a pair of valuations $\langle v_1,v_2\rangle$ just with $v$ if there is no risk of confusion. Furthermore, for each frame $\mathfrak{F}$ and each $w\in\mathfrak{F}$, we denote
\begin{align*}
R(w)&=\{w':wRw'=1\},\tag{for fuzzy frames}\\
R(w)&=\{w':wRw'\}.\tag{for crisp frames}
\end{align*}
\end{convention}
\begin{convention}
We will further denote with $\KbiG$ the set of all formulas $\KbiG$-valid on all \emph{crisp} frames; $\KbiG^\mathsf{f}$ the set of all formulas $\KbiG$-valid on all \emph{fuzzy} frames; and $\KGsquare$ --- the set of all formulas $\KGsquare$ valid on all \emph{crisp} frames.
\end{convention}
Before proceeding to establish some semantical properties, let us make two remarks. First,  neither $\Box$ nor $\lozenge$ are trivialised by contradictions: in contrast to $\mathbf{K}$, $\Box(p\wedge\neg p)\rightarrow\Box q$ is not $\KGsquare$ valid, and neither is $\lozenge(p\wedge\neg p)\rightarrow\lozenge q$. Intuitively, this means that one can have contradictory but non-trivial beliefs. Second, we can formalise statements of comparative belief such as the ones we have already given before:
\begin{quote}
\textsf{wallet}: \emph{I~am more confident that the wallet is Paula's than that the wallet is Quentin's.}\\
\textsf{ice cream}: \emph{Alice is more certain than Britney that Claire loves pistachio ice cream.}
\end{quote}
For this, consider the following defined operators.
\begin{align}
\Delta\tau&\coloneqq{\sim}(\mathbf{1}\coimplies\tau)\label{eq:+topdefinition}\\
\Delta^\neg\phi&\coloneqq{\sim}(\mathbf{1}\coimplies\phi)\wedge\neg{\sim\sim}(\mathbf{1}\coimplies\phi)\label{eq:topdefinition}
\end{align}
It is clear that for any $\tau\in\bimodalL$ and $\phi\in\bimodalLsquare$ interpreted on $\KbiG$ and $\KGsquare$ models, respectively, it holds that
\begin{align}
v(\Delta\tau,w)&=
\begin{cases}
1&\text{ if }v(\tau,w)=1\\
0&\text{ otherwise},
\end{cases}
&&
v(\Delta^\neg\phi,w)&=
\begin{cases}
(1,0)&\text{if }v(\phi,w)=(1,0)\\
(0,1)&\text{otherwise}.
\end{cases}
\end{align}
Now we can define formulas that express order relations between values of two formulas both for $\KbiG$ and $\KGsquare$.

For $\KbiG$ they look as follows:
\begin{align*}
v(\tau,w)\leq v(\tau',w)&\text{ iff } v(\Delta(\tau\rightarrow\tau'),w)=1,\\
v(\tau,w)>v(\tau',w)&\text{ iff } v\left({\sim}\Delta(\tau'\rightarrow\tau),w\right)=1.
\end{align*}
In $\KGsquare$, the orders are defined in a more complicated way:
\begin{align*}
v(\phi,w)\leq v(\phi',w)&\text{ iff } v(\Delta^\neg(\phi\rightarrow\phi'),w)=(1,0),\\
v(\phi,w)>v(\phi',w)&\text{ iff } v(\Delta^\neg(\phi'\rightarrow\phi)\wedge{\sim}\Delta^\neg(\phi\rightarrow\phi'),w)=(1,0).
\end{align*}
Observe, first, that both in $\KbiG$ and $\KGsquare$ the relation ‘the value of $\tau$ ($\phi$) is less or equal to the value of $\tau'$ ($\phi'$)’ is defined as ‘$\tau\rightarrow\tau'$ ($\phi\rightarrow\phi'$) has the designated value’. In $\KbiG$, the strict order is just a negation of the non-strict order since all values are comparable. On the other hand, in contrast to $\KbiG$, the strict order in $\KGsquare$ is not a simple negation of the non-strict order since $\KGsquare$ is essentially two-dimensional. We provide further details in remark~\ref{rem:comparativebelief}.

Finally, we can formalise $\mathsf{wallet}$ as follows. We interpret ‘I am confident’ as $\Box$ and substitute ‘the wallet is Paula's’ with $p$, and ‘the wallet is Quentin's’ with~$q$. Now, we just use the definition of $>$ in $\bimodalLsquare$ to get
\begin{align}
\Delta^\neg(\Box p\rightarrow\Box q)\wedge{\sim}\Delta^\neg(\Box q\rightarrow\Box p).
\label{eq:walletexample}
\end{align}
For \textsf{ice cream}, we need two different modalities: $\Box_a$ and $\Box_b$ for Alice and Brittney, repectively. Replacing ‘Alice loves pistachio ice cream’ with $p$, we get
\begin{align}
\Delta^\neg(\Box_a p\rightarrow\Box_b p)\wedge{\sim}\Delta^\neg(\Box_b p\rightarrow\Box_a p).
\label{eq:icecreamexample}
\end{align}
\begin{remark}
$\Delta$ is called Baaz' delta (cf., e.g.~\cite{Baaz1996} for more details). Intuitively, $\Delta\tau$ can be interpreted as ‘$\tau$ has the designated value’ and acts much like a necessity modality: if $\tau$ is $\KbiG$ valid, then so is $\Delta\tau$; moreover, $\Delta(p\rightarrow q)\rightarrow(\Delta p\rightarrow\Delta q)$ is valid. Furthermore, $\Delta$ and $\coimplies$ can be defined via one another in $\KbiG$, thus the addition of $\Delta$ to $\mathsf{G}$ makes it more expressive and allows to define both strict and non-strict orders.
\end{remark}
\begin{remark}\label{rem:comparativebelief}
Recall that we mentioned in section~\ref{sec:introduction} that an agent should usually be able to compare their beliefs in different statements: this is reflected by the fact that $\Delta(p\rightarrow q)\vee\Delta(q\rightarrow p)$
is $\KbiG$ valid. It can be counter-intuitive if the contents of beliefs have nothing in common, however.

This drawback is avoided if we treat support of truth and support of falsity independently. Here is where a difference between $\KbiG$ and $\KGsquare$ lies. In $\KGsquare$, we can \emph{only compare the values of formulas coordinate-wise}, whence
$\Delta^\neg(p\rightarrow q)\vee\Delta^\neg(q\rightarrow p)$ is not $\KGsquare$ valid. E.g., if we set $v(p,w)=(0.7,0.6)$ and $v(q,w)=(0.4,0.2)$, $v(p,w)$ and $v(q,w)$ will not be comparable w.r.t.\ the truth (upward) order on $[0,1]^\Join$.
\end{remark}
We end this section with establishing some useful semantical properties.
\begin{proposition}\label{prop:+isenough}
$\mathfrak{F}\models_{\KGsquare}\!\phi$ iff for any model $\mathfrak{M}$ on $\mathfrak{F}$ and any $w\!\in\!\mathfrak{F}$, $v_1(\phi,w)\!=\!1$.
\end{proposition}
\begin{proof}
The ‘if’ direction is evident from the definition of validity. We show the ‘only if’ part. It suffices to show that the following statement holds for any $\phi$ and $w\in\mathfrak{F}$:
\begin{quote}
\emph{for any $v(p,w)=(x,y)$, let $v^*(p,w)=(1-y,1-x)$. Then $v(\phi,w)=(x,y)$ iff $v^*(\phi,w)=(1-y,1-x)$.}
\end{quote}
We proceed by induction on $\phi$. The proof of propositional cases is identical to the one in~\cite[Proposition~5]{BilkovaFrittellaKozhemiachenko2021}. We consider only the case of $\phi=\Box\psi$ since $\Box$ and $\lozenge$ are interdefinable.

Let $v(\Box\psi,w)=(x,y)$. Then $\inf\{v_1(\psi,w'):wRw'\}=x$, and $\sup\{v_2(\psi,w'):wRw'\}=y$. Now, we apply the induction hypothesis to $\psi$, and thus if $v(\psi,s)=(x',y')$, then $v^*(\psi,s)=(1-y',1-x')$ for any $s\in R(w)$. But then $\inf\{v^*_1(\psi,w'):wRw'\}=1-y$, and $\sup\{v^*_2(\psi,w'):wRw'\}=1-x$ as required.

Now, assume that $v_1(\phi,w)=1$ for any $v_1$ and $w$. We can show that $v_2(\phi,w)\!=\!0$ for any $w$ and $v_2$. Assume for contradiction that $v_2(\phi,w)\!=\!y\!>\!0$ but $v_1(\phi,w)\!=\!1$. Then, $v^*(\phi)\!=\!(1\!-\!y,1\!-\!1)\!=\!(1\!-\!y,0)$. But since $y\!>\!0$, $v^*(\phi)\!\neq\!(1,0)$.
\end{proof}
\begin{proposition}\label{prop:conservativity}
\begin{enumerate}
\item[]
\item Let $\phi$ be a formula over $\{\mathbf{0},\wedge,\vee,\rightarrow,\Box,\lozenge\}$. Then, $\mathfrak{F}\models_{\KG}\phi$ iff $\mathfrak{F}\models_{\KbiG^\mathsf{f}}\phi$ and $\mathfrak{F}\models_{\KG^c}\phi$ iff $\mathfrak{F}\models_{\KbiG}\phi$, for any~$\mathfrak{F}$.
\item Let $\phi\in\bimodalL$. Then, $\mathfrak{F}\models_{\KbiG}\phi$ iff $\mathfrak{F}\models_{\KGsquare}\phi$, for any crisp~$\mathfrak{F}$.
\end{enumerate}
\end{proposition}
\begin{proof}
1. follows directly from the semantic conditions of definition~\ref{def:KbiGsemantics}. We consider 2. The ‘only if’ direction is straightforward since the semantic conditions of $v_1$ in $\KGsquare$ models and $v$ in $\KbiG$ models coincide. The ‘if’ direction follows from proposition~\ref{prop:+isenough}: if $\phi$ is valid on $\mathfrak{F}$, then $v(\phi,w)=1$ for any $w\in\mathfrak{F}$ and any $v$ on $\mathfrak{F}$. But then, $v_1(\phi,w)=1$ for any $w\in\mathfrak{F}$. Hence, $\mathfrak{F}\models_{\KGsquare}\phi$.
\end{proof}
\section{Model-theoretic properties of $\KGsquare$}\label{sec:modeltheory}
In the previous section, we have seen how the addition of $\coimplies$ allowed us to formalise statements considering comparison of beliefs. Here, we will show that both $\Box$ and $\lozenge$ fragments of $\KbiG$, and hence $\KGsquare$, are strictly more expressive than the classical modal logic $\mathbf{K}$, i.e. that they can define all classically definable classes of crisp frames as well as some undefinable ones.
\begin{definition}[Frame definability]
Let $\Sigma$ be a set of formulas. $\Sigma$ \emph{defines} a~class of frames $\mathbb{K}$ in a logic $\mathbf{L}$ iff it holds that $\mathfrak{F}\in\mathbb{K}$ iff $\mathfrak{F}\models_\mathbf{L}\Sigma$.
\end{definition}
The next statement follows from proposition~\ref{prop:conservativity} since $\mathbf{K}$ can be faithfully embedded in $\KG^c$ by substituting each variable $p$ with ${\sim\sim}p$ (cf.~\cite{MetcalfeOlivetti2009,MetcalfeOlivetti2011} for details).
\begin{theorem}\label{theorem:definability}
Let $\mathbb{K}$ be a class of frames definable in $\mathbf{K}$. Then, $\mathbb{K}$ is definable in $\KbiG$ and $\KGsquare$.
\end{theorem}
\begin{theorem}\label{theorem:finitebranching}
\begin{enumerate}
\item Let $\mathfrak{F}$ be \emph{crisp}. Then $\mathfrak{F}$ is finitely branching (i.e., $R(w)$ is finite for every $w\in\mathfrak{F}$) iff $\mathfrak{F}\models_{\KbiG}\mathbf{1}\coimplies\lozenge((p\coimplies q)\wedge q)$.
\item Let $\mathfrak{F}$ be \emph{fuzzy}. Then $\mathfrak{F}$ is finitely branching and $\sup\{wRw':wRw'<1\}<1$ for all $w\in\mathfrak{F}$ iff $\mathfrak{F}\models_{\KbiG}\mathbf{1}\coimplies\lozenge((p\coimplies q)\wedge q)$.
\end{enumerate}
\end{theorem}
\begin{proof}
We show the case of fuzzy frames since the crisp ones can be tackled in the same manner. Assume that $\mathfrak{F}$ is finitely branching and that $\sup\{wRw'\!:\!wRw'\!<\!1\}<1$ for all $w\in\mathfrak{F}$. It suffices to show that $v(\lozenge((p\coimplies q)\wedge q),w)<1$ for all $w\in\mathfrak{F}$. First of all, observe that there is no $w'\in\mathfrak{F}$ s.t.\ $v((p\coimplies q)\wedge q,w')=1$. It is clear that $\sup\limits_{wRw'<1}\{v((p\coimplies q)\wedge q,w')\wedge_\mathsf{G}wRw'\}<1$ and that
$$\sup\{v((p\coimplies q)\wedge q,w'):wRw'=1\}=\max\{v((p\coimplies q)\wedge q,w'):wRw'=1\}<1$$
since $R(w)$ is finite. But then $v(\lozenge((p\coimplies q)\wedge q),w)<1$ as required.

For the converse, either (1) $R(w)$ is infinite for some $w$, or (2) $\sup\{wRw':wRw'<1\}=1$ for some $w$. For (1), set $v(p,w')=1$ for every $w'\in R(w)$. Now let $W'\subseteq R(w)$ and $W'=\{w_i:i\in\{1,2,\ldots\}\}$. We set $v(q,w_i)=\frac{i}{i+1}$. It is easy to see that $\sup\{v(q,w_i):w_i\in W'\}=1$ and that $v((p\coimplies q)\wedge q,w_i)=v(q,w_i)$. Therefore, $v(\mathbf{1}\coimplies\lozenge((p\coimplies q)\wedge q),w)=0$.

For (2), we let $v(p,w')=1$ and further, $v(q,w')=wRw'$ for all $w'\in\mathfrak{F}$. Now since $\sup\{wRw':wRw'<1\}=1$ and $v(((p\coimplies q)\wedge q),w')=v(q,w')$ for all $w'\in\mathfrak{F}$, it follows that $v(\lozenge((p\coimplies q)\wedge q),w)=1$, whence $v(\mathbf{1}\coimplies\lozenge((p\coimplies q)\wedge q),w)=0$.
\end{proof}
\begin{remark}\label{rem:noGlivenko}
The obvious corollary of theorem~\ref{theorem:finitebranching} is the lack of FMP for the $\lozenge$-fragment of $\KbiG^\mathsf{f}$\footnote{Bi-modal $\KbiG^\mathsf{f}$ lacks have FMP since it is a conservative extension of $\KG$.} since $\lozenge((p\coimplies q)\wedge q)$ in never true in a finite model. This differentiates $\KbiG^\mathsf{f}$ from $\KG$ since the $\lozenge$-fragment of $\KG$ \emph{has} FMP~\cite[Theorem~7.1]{CaicedoRodriguez2010}. Moreover, one can define finitely branching frames in $\Box$ fragments of $\KG$ and $\KG^c$. Indeed, ${\sim\sim}\Box(p\vee{\sim}p)$ serves as such definition.
\end{remark}
\begin{corollary}
$\KGsquare$ and both $\Box$ and $\lozenge$ fragments of $\KbiG$ are strictly more expressive than $\mathbf{K}$.
\end{corollary}
\begin{proof}
From theorems~\ref{theorem:definability} and~\ref{theorem:finitebranching} since $\mathbf{K}$ is complete both w.r.t.\ all frames and all finitely branching frames. The result for $\KGsquare$ follows since it is conservative over $\KbiG$ (proposition~\ref{prop:conservativity}).
\end{proof}
These results show us that addition of $\coimplies$ greatly enhances the expressive power of our logic. Here it is instructive to remind ourselves that classical epistemic logics are usually complete w.r.t.\ finitely branching frames (cf.~\cite{FaginHalpernMosesVardi2003} for details). It is reasonable since for practical reasoning, agents cannot consider infinitely many alternatives. In our case, however, if we wish to use $\KbiG$ and $\KGsquare$ for knowledge representation, we need to \emph{impose} finite branching explicitly.


Furthermore, allowing for infinitely branching frames in $\KbiG$ or $\KGsquare$ leads to counter-intuitive consequences. In particular, it is possible that $v(\Box\phi,w)=(0,1)$ even though there are no $w',w''\in R(w)$ s.t. $v_1(\phi,w')=0$ or $v_2(\phi,w'')=1$. In other words, there is no source that decisively falsifies $\phi$, furthermore, all sources have some evidence \emph{for} $\phi$, and yet we somehow believe that $\phi$ is completely false and untrue. Dually, it is possible that $v(\lozenge\phi,w)=(1,0)$ although there are no $w',w''\in R(w)$ s.t. $v_1(\phi,w')=1$ or $v_2(\phi,w'')=0$. Even though $\lozenge$ is an ‘optimistic’ aggregation, it should not ignore the fact that \emph{all} sources have some evidence \emph{against} $\phi$ but \emph{none} supports it completely.

Of course, this situation is impossible if we consider only finitely branching frames for infima and suprema will become minima and maxima. There, all values of modal formulas will be \emph{witnessed} by some accessible states in the following sense. For $\heartsuit\in\{\Box,\lozenge\}$, $i\in\{1,2\}$, if $v_i(\heartsuit\phi,w)=x$, then there is $w'\in R(w)$ s.t. $v_i(\phi,w')=x$. Intuitively speaking, finitely branching frames represent the situation when our degree of certainty in some statement is based uniquely on the data given by the sources.
\begin{convention}
We will further use $\fbKbiG$ and $\fbKGsquare$ to denote the sets of all $\bimodalL$ and $\bimodalLsquare$ formulas valid on finitely branching crisp frames.
\end{convention}
Observe, moreover, that $\Box$ and $\lozenge$ are still undefinable via one another in $\bimodalL$. The proof is the same as that of~\cite[Lemma~6.1]{RodriguezVidal2021}.
\begin{proposition}\label{prop:+noninterdefinability}
$\Box$ and $\lozenge$ are not interdefinable in $\fbKbiG$.
\end{proposition}
\begin{corollary}
\begin{enumerate}
\item[]
\item $\Box$ and $\lozenge$ are not interdefinable in $\KbiG$, $\fbKbiG^\mathsf{f}$, and $\KbiG^\mathsf{f}$.
\item Both $\Box$ and $\lozenge$ fragments of $\KbiG$ are more expressive than $\mathbf{K}$.
\end{enumerate}
\end{corollary}
In the remainder of the paper, we are going to provide a complete proof system for $\fbKGsquare$ (and hence, $\fbKbiG$), and establish its decidability and complexity as well as finite model property. Note, however, that the latter is not entirely for granted. In fact, several expected ways of defining filtration (cf.~\cite{ChagrovZakharyaschev1997,BlackburndeRijkeVenema2010} for more details thereon) fail.

Let $\Sigma\subseteq\bimodalL$ be closed under subformulas. If we want to have filtration for $\fbKbiG$, there are three intuitive ways to define $\sim_\Sigma$ on the carrier of a~model that is supposed to relate states satisfying the same formulas.
\begin{enumerate}
\item $w\sim^1_\Sigma w'$ iff $v(\phi,w)=v(\phi,w')$ for all $\phi\in\Sigma$.
\item $w\sim^2_\Sigma w'$ iff $v(\phi,w)=1\Leftrightarrow v(\phi,w')=1$ for all $\phi\in\Sigma$.
\item $w\sim^3_\Sigma w'$ iff $v(\phi,w)\leq v(\phi',w)\Leftrightarrow v(\phi,w')\leq v(\phi',w')$ for all $\phi,\phi'\!\in\!\Sigma\!\cup\!\{\mathbf{0},\!\mathbf{1}\}$.
\end{enumerate}
Consider the model on fig.~\ref{fig:filtrationcounterexample} and two formulas:
\begin{align*}
\phi^\leq&\coloneqq{\sim\sim}(p\rightarrow\lozenge p)
&
\phi^>&\coloneqq{\sim\sim}(p\coimplies\lozenge p)
\end{align*}
Now let $\Sigma$ to be the set of all subformulas of $\phi^\leq\wedge\phi^>$.
\begin{figure}
\[\mathfrak{M}:\xymatrix{w_1\ar[r]&w_2\ar[r]&\ldots\ar[r]&w_n\ar[r]&\ldots}\]
\caption{$v(p,w_n)=\frac{1}{n+1}$}
\label{fig:filtrationcounterexample}
\end{figure}

First of all, it is clear that $v(\phi^\leq\wedge\phi^>,w)=1$ for any $w\in\mathfrak{M}$. Observe now that all states in $\mathfrak{M}$ are \emph{distinct} w.r.t. $\sim^1_\Sigma$. Thus, the first way of constructing the carrier of the new model does not give the FMP.

As regards to $\sim^2_\Sigma$ and $\sim^3_\Sigma$, one can check that for any $w,w'\in\mathfrak{M}$, it holds that $w\sim^2_\Sigma w'$ and $w\sim^3_\Sigma w'$. So, if we construct a filtration of $\mathfrak{M}$ using equivalence classes of either of these two relations, the carrier of the resulting model is going to be finite. Even more so, it is going to be a singleton.

However, we can show that there is \emph{no finite model} $\mathfrak{N}=\langle U,S,e\rangle$ s.t.
\[\forall s\in\mathfrak{N}:v(\phi^\leq\wedge\phi^>,s)=1.\]
Indeed, $e(\phi^\leq,t)=1$ iff $e(p,t')>0$ for some $t'\in S(t)$, while $e(\phi^>,t)=1$ iff $v(p,t)>v(p,t')$ for any $t'\in S(t)$. Now, if $U$ is finite, we have two options: either (1) there is $u\in U$ s.t. $R(u)=\varnothing$, or (2) $U$ contains a finite $S$-cycle.

For (1), note that $v(\lozenge p,u)=0$, and we have two options: if $e(p,u)=0$, then $e(\phi^>,u)=0$; if, on the other hand, $e(p,u)>0$, then $e(\phi^\leq,u)=0$. For (2), assume w.l.o.g. that the $S$-cycle looks as follows: $u_0Su_1Su_2\ldots Su_nSu_0$.

If $e(p,u_0)\!=\!0$, $e(\phi^>,u_0)\!=\!0$, so $e(p,u_0)\!>\!0$. Furthermore, $e(p,u_i)\!>\!e(p,u_{i+1})$. Otherwise, again, $e(\phi^>,u_i)=0$. But then we have $e(\phi^>,u_i)=0$.

But this means that $\sim^2_\Sigma$ and $\sim^3_\Sigma$ do not preserve truth of formulas from $w$ to $[w]_\Sigma$, i.e., neither of these two relations can be used to define filtration. Thus, in order to explicitly prove the finite model property and establish complexity evaluations for $\fbKbiG$ and $\fbKGsquare$, we will provide a tableaux calculus. It will also serve as a decision procedure for satisfiability and validity of formulas.
\section{Tableaux for $\fbKGsquare$}\label{sec:tableaux}
Usually, proof theory for modal and many-valued logics is presented in one of the following several forms. The first one is a Hilbert-style axiomatisation as given in e.g.~\cite{Hajek1998} for the propositional G\"{o}del logic and in~\cite{CaicedoRodriguez2010,CaicedoRodriguez2015,RodriguezVidal2021} for its modal expansions. Hilbert calculi are useful for establishing frame correspondence results as well as for showing that one logic extends another one in the same language. On the other hand, their completeness proofs might be quite complicated, and the proof-search not at all straightforward. Second, there are non-labelled sequent and hyper-sequent calculi (cf.~\cite{MetcalfeOlivettiGabbay2008} for the propositional proof systems and~\cite{MetcalfeOlivetti2009,MetcalfeOlivetti2011} for the modal hypersequent calculi). With regards to modal logics, completeness proofs of (hyper)sequent calculi often provide the answer for the decidability problem. Furthermore, the proof search can be quite straightforwardly automatised provided that the calculus is \emph{cut-free}.

Finally, there are proof systems that directly incorporate semantics: in particular, tableaux (e.g., the ones for G\"{o}del logics~\cite{AvronKonikowska2001} and tableaux for \L{}ukasiewicz description logic~\cite{KulackaPattinsonSchroeder2013}) and labelled sequent calculi (cf., e.g.~\cite{Negri2005} for labelled sequent calculi for classical modal logics). Because of the calculi's nature, their completeness proofs are usually simple. Besides, the calculi serve as a decision procedure that either establishes that the given formula is valid or provides an explicit countermodel.

Our tableaux system $\mathcal{T}\!\left(\fbKGsquare\right)$ is a~straightforward modal expansion of constraint tableaux for $\Gsquare$ presented in~\cite{BilkovaFrittellaKozhemiachenko2021}. It is inspired by constraint tableaux for \L{}u\-ka\-sie\-wicz logics from~\cite{Haehnle1992,Haehnle1994} (but cf.~\cite{diLascioGisolfi2005} for an approach similar to ours) which we modify with two-sorted labels corresponding to the support of truth and support of falsity in the model. This idea comes from tableaux for the Belnap --- Dunn logic by D'Agostino~\cite{DAgostino1990}. Moreover, since $\fbKGsquare$ is a~conservative extension of $\fbKbiG$, our calculus can be used for that logic as well if we apply only the rules that govern the support of truth of $\bimodalL$ formulas.
\begin{definition}[$\mathcal{T}\!\left(\fbKGsquare\right)$]\label{def:KG2constrainttableau} We fix a set of state-labels $\mathsf{W}$ and let $\lesssim\in\!\{<,\leqslant\}$ and $\gtrsim\in\!\{>,\geqslant\}$. Let further $w\!\in\!\mathsf{W}$, $\mathbf{x}\!\in\!\{1,2\}$, $\phi\!\in\!\bimodalLsquare$, and $c\!\in\!\{0,1\}$. A~structure is either $w\!:\!\mathbf{x}\!:\!\phi$ or $c$. We denote the set of structures with $\mathsf{Str}$.

We define a \emph{constraint tableau} as a downward branching tree whose branches are sets containing the following types of entries:
\begin{itemize}[noitemsep,topsep=2pt]
\item \emph{relational constraints} of the form $w\mathsf{R}w'$ with $w,w'\in\mathsf{W}$;
\item \emph{structural constraints} of the form $\mathfrak{X}\lesssim\mathfrak{X}'$ with $\mathfrak{X},\mathfrak{X}'\in\mathsf{Str}$.
\end{itemize}
Each branch can be extended by an application of a~rule\footnote{If $\mathfrak{X}<1$ and $\mathfrak{X}<\mathfrak{X}'$ (or $0<\mathfrak{X}'$ and $\mathfrak{X}<\mathfrak{X}'$) occur on $\mathcal{B}$, then the rules are applied only to $\mathfrak{X}<\mathfrak{X}'$.} from fig.~\ref{fig:G2propositionalrules} or fig.~\ref{fig:modalrules}.
\begin{figure}
\centering
\[\begin{array}{cccc}
\neg_1\!\lesssim\!\dfrac{w\!:\!1\!:\!\neg\phi\!\lesssim\!\mathfrak{X}}{w\!:\!2\!:\!\phi\!\lesssim\!\mathfrak{X}}
& \quad
\neg_2\!\lesssim\!\dfrac{w\!:\!2\!:\!\neg\phi\!\lesssim\!\mathfrak{X}}{w\!:\!1\!:\!\phi\!\lesssim\!\mathfrak{X}}
& \quad
\neg_1\!\gtrsim\!\dfrac{w\!:\!1\!:\!\neg\phi\!\gtrsim\!\mathfrak{X}}{w\!:\!2\!:\!\phi\!\gtrsim\!\mathfrak{X}}
& \quad
\neg_2\!\gtrsim\!\dfrac{w\!:\!2\!:\!\neg\phi\!\gtrsim\!\mathfrak{X}}{w\!:\!1\!:\!\phi\!\gtrsim\!\mathfrak{X}}
\end{array}\]
\smallskip
\[\begin{array}{cccc}
\wedge_1\!\gtrsim\!\dfrac{w\!:\!1\!:\!\phi\!\wedge\!\phi'\!\gtrsim\!\mathfrak{X}}{\begin{matrix}w\!:\!1\!:\!\phi\!\gtrsim\!\mathfrak{X}\\w\!:\!1\!:\!\phi'\!\gtrsim\!\mathfrak{X}\end{matrix}}
&
\wedge_2\!\lesssim\!\dfrac{w\!:\!2\!:\!\phi\!\wedge\!\phi'\!\lesssim\!\mathfrak{X}}{\begin{matrix}w\!:\!2\!:\!\phi\!\lesssim\!\mathfrak{X}\\w\!:\!2\!:\!\phi'\!\lesssim\!\mathfrak{X}\end{matrix}}
&
\vee_1\!\lesssim\!\dfrac{w\!:\!1\!:\!\phi\!\vee\!\phi'\!\lesssim\!\mathfrak{X}}{\begin{matrix}w\!:\!1\!:\!\phi\!\lesssim\!\mathfrak{X}\\w\!:\!1\!:\!\phi'\!\lesssim\!\mathfrak{X}\end{matrix}}
&
\vee_2\!\gtrsim\!\dfrac{w\!:\!2\!:\!\phi\!\vee\!\phi'\!\gtrsim\!\mathfrak{X}}{\begin{matrix}w\!:\!2\!:\!\phi\!\gtrsim\!\mathfrak{X}\\w\!:\!2\!:\!\phi'\!\gtrsim\!\mathfrak{X}\end{matrix}}
\end{array}\]
\[\begin{array}{cc}
\wedge_1\!\lesssim\!\dfrac{w\!:\!1\!:\!\phi\wedge\phi'\!\lesssim\!\mathfrak{X}}{w\!:\!1\!:\!\phi\!\lesssim\!\mathfrak{X}\mid w\!:\!1\!:\!\phi'\!\lesssim\!\mathfrak{X}}
&\quad
\wedge_2\!\gtrsim\!\dfrac{w\!:\!2\!:\!\phi\wedge\phi'\!\gtrsim\!\mathfrak{X}}{w\!:\!2\!:\!\phi\!\gtrsim\!\mathfrak{X}\mid w\!:\!2\!:\!\phi'\!\gtrsim\!\mathfrak{X}}\\
&\\
\vee_1\!\gtrsim\!\dfrac{w\!:\!1\!:\!\phi\vee\phi'\!\gtrsim\!\mathfrak{X}}{w\!:\!1\!:\!\phi\!\gtrsim\!\mathfrak{X}\mid w\!:\!1\!:\!\phi'\!\gtrsim\!\mathfrak{X}}
& \qquad
\vee_2\!\lesssim\!\dfrac{w\!:\!2\!:\!\phi\vee\phi'\!\lesssim\!\mathfrak{X}}{w\!:\!2\!:\!\phi\!\lesssim\!\mathfrak{X}\mid w\!:\!2\!:\!\phi'\!\lesssim\!\mathfrak{X}}\\
\end{array}\]
\smallskip
\[\begin{array}{ccc}
\rightarrow_1\!\leqslant\!\dfrac{w\!:\!1\!:\!\phi\rightarrow\phi'\!\leqslant\!\mathfrak{X}}{\mathfrak{X}\!\geqslant\!{1}\left|\begin{matrix}\mathfrak{X}\!<\!{1}\\w\!:\!1\!:\!\phi'\!\leqslant\!\mathfrak{X}\\w\!:\!1\!:\!\phi\!>\!w\!:\!1\!:\!\phi'\end{matrix}\right.}&\rightarrow_1\!\gtrsim\!\dfrac{w\!:\!1\!:\!\phi\rightarrow\phi'\!\gtrsim\!\mathfrak{X}}{w\!:\!1\!:\!\phi\!\leqslant\!w\!:\!1\!:\!\phi'\mid w\!:\!1\!:\!\phi'\!\gtrsim\!\mathfrak{X}}\\
&\\
\rightarrow_2\!\lesssim\!\dfrac{w\!:\!2\!:\!\phi\rightarrow\phi'\!\lesssim\!\mathfrak{X}}{w\!:\!2\!:\!\phi'\!\leqslant\!w\!:\!2\!:\!\phi\mid w\!:\!2\!:\!\phi'\!\lesssim\!\mathfrak{X}}&\rightarrow_2\!\geqslant\!\dfrac{w\!:\!2\!:\!\phi\rightarrow\phi'\!\geqslant\!\mathfrak{X}}{\mathfrak{X}\!\leqslant\!{0}\left|\begin{matrix}\mathfrak{X}\!>\!{0}\\w\!:\!2\!:\!\phi'\!\geqslant\!\mathfrak{X}\\w\!:\!2\!:\!\phi'\!>\!w\!:\!2\!:\!\phi\end{matrix}\right.}\\
&\\
\Yleft_1\!\lesssim\!\dfrac{w\!:\!1\!:\!\phi\!\Yleft\!\phi'\!\lesssim\!\mathfrak{X}}{w\!:\!1\!:\!\phi\!\leqslant\!w\!:\!1\!:\!\phi'\mid w\!:\!1\!:\!\phi\!\lesssim\!\mathfrak{X}}&\Yleft_1\!\geqslant\!\dfrac{w\!:\!1\!:\!\phi\!\Yleft\!\phi'\!\geqslant\!\mathfrak{X}}{\mathfrak{X}\!\leqslant\!{0}\left|\begin{matrix}\mathfrak{X}\!>\!{0}\\w\!:\!1\!:\!\phi\!\geqslant\!\mathfrak{X}\\w\!:\!1\!:\!\phi\!>\!w\!:\!1\!:\!\phi'\end{matrix}\right.}\\
&\\
\Yleft_2\!\gtrsim\!\dfrac{w\!:\!2\!:\!\phi\Yleft\phi'\!\gtrsim\!\mathfrak{X}}{w\!:\!2\!:\!\phi\!\gtrsim\!\mathfrak{X}\mid w\!:\!2\!:\!\phi'\!\leqslant\!w\!:\!2\!:\!\phi}&\Yleft_2\!\leqslant\!\dfrac{w\!:\!2\!:\!\phi\Yleft\phi'\!\leqslant\!\mathfrak{X}}{\mathfrak{X}\!\geqslant\!{1}\left|\begin{matrix}\mathfrak{X}\!<\!{1}\\w\!:\!2\!:\!\phi\!\leqslant\!\mathfrak{X}\\w\!:\!2\!:\!\phi'\!>\!w\!:\!2\!:\!\phi\end{matrix}\right.}
\end{array}\]
\[\begin{array}{cc}
\rightarrow_1\!<\!\dfrac{w\!:\!1\!:\!\phi\rightarrow\phi'\!<\!\mathfrak{X}}{\begin{matrix}w\!:\!1\!:\!\phi'\!<\!\mathfrak{X}\\w\!:\!1\!:\!\phi\!>\!w\!:\!1\!:\!\phi'\end{matrix}}
&
\rightarrow_2\!>\!\dfrac{w\!:\!2\!:\!\phi\rightarrow\phi'\!>\!\mathfrak{X}}{\begin{matrix}w\!:\!2\!:\!\phi'\!>\!\mathfrak{X}\\w\!:\!2\!:\!\phi'\!>\!w\!:\!2\!:\!\phi\end{matrix}}
\end{array}\]
\smallskip
\[\begin{array}{cc}
\Yleft_1\!>\!\dfrac{w\!:\!1\!:\!\phi\!\Yleft\!\phi'\!>\!\mathfrak{X}}{\begin{matrix}w\!:\!1\!:\!\phi\!>\!\mathfrak{X}\\w\!:\!1\!:\!\phi\!>\!w\!:\!1\!:\!\phi'\end{matrix}}
&
\Yleft_2\!<\!\dfrac{w\!:\!2\!:\!\phi\Yleft\phi'\!<\!\mathfrak{X}}{\begin{matrix}w\!:\!2\!:\!\phi\!<\!\mathfrak{X}\\w\!:\!2\!:\!\phi\!<\!w\!:\!2\!:\!\phi'\end{matrix}}
\end{array}\]
\caption{Propositional rules of $\mathcal{T}\left(\fbKGsquare\right)$. Bars denote branching.}
\label{fig:G2propositionalrules}
\end{figure}
\begin{figure}
\[\begin{array}{cccc}
\Box_1\!\gtrsim\dfrac{\begin{matrix}w\!:\!1\!:\!\Box\phi\gtrsim\mathfrak{X}\\wRw'\end{matrix}}{\begin{matrix}w'\!:\!1\!:\!\phi\gtrsim\mathfrak{X}\end{matrix}}&\Box_1\!\lesssim\dfrac{\begin{matrix}w\!:\!1\!:\!\Box\phi\lesssim\mathfrak{X}\end{matrix}}{\begin{matrix}wRw''\\w''\!:\!1\!:\!\phi\lesssim\mathfrak{X}\end{matrix}}&\Box_2\!\gtrsim\dfrac{\begin{matrix}w\!:\!2\!:\!\Box\phi\gtrsim\mathfrak{X}\end{matrix}}{\begin{matrix}wRw''\\w''\!:\!2\!:\!\phi\gtrsim\mathfrak{X}\end{matrix}}&\Box_2\!\lesssim\dfrac{\begin{matrix}w\!:\!2\!:\!\Box\phi\lesssim\mathfrak{X}\\wRw'\end{matrix}}{\begin{matrix}w'\!:\!2\!:\!\phi\lesssim\mathfrak{X}\end{matrix}}
\\
\lozenge_1\!\lesssim\dfrac{\begin{matrix}w\!:\!1\!:\!\lozenge\phi\lesssim\mathfrak{X}\\wRw'\end{matrix}}{\begin{matrix}w'\!:\!1\!:\!\phi\lesssim\mathfrak{X}\end{matrix}}&\lozenge_1\!\gtrsim\dfrac{\begin{matrix}w\!:\!1\!:\!\lozenge\phi\gtrsim\mathfrak{X}\end{matrix}}{\begin{matrix}wRw''\\w''\!:\!1\!:\!\phi\gtrsim\mathfrak{X}\end{matrix}}&\lozenge_2\!\lesssim\dfrac{\begin{matrix}w\!:\!2\!:\!\lozenge\phi\lesssim\mathfrak{X}\end{matrix}}{\begin{matrix}wRw''\\w''\!:\!2\!:\!\phi\lesssim\mathfrak{X}\end{matrix}}&\lozenge_2\!\gtrsim\dfrac{\begin{matrix}w\!:\!2\!:\!\lozenge\phi\gtrsim\mathfrak{X}\\wRw'\end{matrix}}{\begin{matrix}w'\!:\!2\!:\!\phi\gtrsim\mathfrak{X}\end{matrix}}
\end{array}\]
\caption{Modal rules of $\mathcal{T}\left(\fbKGsquare\right)$. $w''$ is fresh on the branch.}
\label{fig:modalrules}
\end{figure}

A tableau's branch $\mathcal{B}$ is \emph{closed} iff one of the following conditions applies:
\begin{itemize}[noitemsep,topsep=2pt]
\item the transitive closure of $\mathcal{B}$ under $\lesssim$ contains $\mathfrak{X}<\mathfrak{X}$;
\item ${0}\geqslant{1}\in\mathcal{B}$, or $\mathfrak{X}>{1}\in\mathcal{B}$, or $\mathfrak{X}<{0}\in\mathcal{B}$.
\end{itemize}
A tableau is \emph{closed} iff all its branches are closed. We say that there is a \emph{tableau proof} of $\phi$ iff there is a closed tableau starting from the constraint $w\!:\!1\!:\!\phi<1$.

An open branch $\mathcal{B}$ is \emph{complete} iff the following condition is met.
\begin{itemize}
\item[$*$]\emph{If all premises of a rule occur on $\mathcal{B}$, then its one conclusion\footnote{Note that branching rules have \emph{two} conclusions.} occurs on~$\mathcal{B}$.}
\end{itemize}
\end{definition}
\begin{remark}
Note that due to proposition~\ref{prop:+isenough}, we need to check only one valuation of $\phi$ to verify its validity.
\end{remark}
\begin{convention}[Interpretation of constraints]\label{conv:TG2meaning}
The following table gives the interpretations of structural constraints on the example of $\leqslant$.
\begin{center}
\begin{tabular}{c|c}
\textbf{entry}&\textbf{interpretation}\\\hline
$w\!:1\!:\!\phi\leqslant w'\!:2\!:\!\phi'$&$v_1(\phi,w)\leq v_2(\phi',w')$\\
$w\!:\!2\!:\!\phi\leqslant c$&$v_2(\phi,w)\leq c$ with $c\in\{0,1\}$
\end{tabular}
\end{center}
\end{convention}
As one can see from fig.~\ref{fig:G2propositionalrules} and fig.~\ref{fig:modalrules}, the rules follow the semantical conditions from definition~\ref{def:KG2semantics}. Let us discuss $\rightarrow_1\!\leqslant$ and $\Box_1\!\lesssim$ in more details.

The premise of $\rightarrow_1\!\leqslant$ is interpreted as $v_1(\phi\rightarrow\phi',w)\leqslant x$. To decompose the implication, we check two options: either $x=1$ (then, the value of $\phi\rightarrow\phi'$ is arbitrary) or $x<1$. In the second case, we use the semantics to obtain that $v_1(\phi',w)\leqslant x$ and $v_1(\phi,w)>v_1(\phi',w)$. 

In order to apply $\Box_1\!\lesssim$ to $w\!:\!1\!:\!\Box\phi\lesssim\mathfrak{X}$, we introduce a new state $w''$ that is seen by $w$. Since we work in a finite branching model, $w''$ can witness the value of $\Box\phi$. Thus, we add $w''\!:\!1\!:\!\phi\lesssim\mathfrak{X}$.

We also provide an example of how our tableaux work. On fig.~\ref{fig:tableauxproofs}, one can see a successful proof on the left and a failed proof on the right.
\begin{figure}[h!]
\footnotesize{
\begin{forest}
smullyan tableaux
[w_0\!:\!1\!:\!\mathbf{1}\!\coimplies\!\lozenge((p\!\coimplies\!q)\!\wedge\!q)\!<\!1
[w_0\!:\!1\!:\!\mathbf{1}\leqslant w_0\!:\!1\!:\!\lozenge((p\!\coimplies\!q)\!\wedge\!q)
[w_0\!:\!1\!:\!\lozenge((p\!\coimplies\!q)\!\wedge\!q)\!\geqslant\!1
[w_0Rw_1[w_1\!:\!1\!:\!(p\!\coimplies\!q)\!\wedge\!q\!\geqslant\!1[w_1\!:\!1\!:\!p\coimplies q\geqslant1[1\leqslant0,closed][w_1\!:\!1\!:\!q\geqslant1[w_1\!:\!1\!:\!p\geqslant1[w_1\!:\!1\!:\!q<1,closed]]]]]]]]
[w_0\!:\!1\!:\!\mathbf{1}<1,closed]]
\end{forest}
}
\quad\quad\quad
\footnotesize{
\begin{forest}
smullyan tableaux
[w_0\!:\!1\!:\!\Box p\rightarrow\Box\Box p<1[w_0\!:\!1\!:\!\Box\Box p\!<\!1[w_0\!:\!1\!:\!\Box p>w_0\!:\!1\!:\!\Box\Box p[w_0Rw_1[w_0\!:\!1\!:\!\Box p>w_1\!:\!1\!:\!\Box p[w_1\!:\!1\!:\!p>w_1\!:\!1\!:\!\Box p[w_1Rw_2[w_1\!:\!1\!:\!p>w_2\!:\!1\!:\!p[\frownie[[[[]]]]]]]]]]]]]
\end{forest}
}
\caption{$\times$ indicates closed branches; $\frownie$ indicates complete open branches.}
\label{fig:tableauxproofs}
\end{figure}
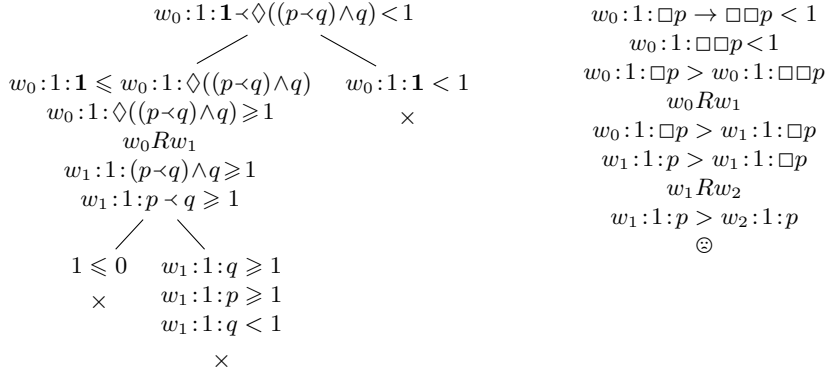
\begin{definition}[Branch realisation]\label{G2branchsatisfaction}
We say that a model $\mathfrak{M}=\langle W,R,v_1,v_2\rangle$ with $W=\{w:w\text{ occurs on }\mathcal{B}\}$ and $R=\{\langle w,w'\rangle:w\mathsf{R}w'\in\mathcal{B}\}$ \emph{realises a~branch $\mathcal{B}$} of a tree iff the following conditions are met.
\begin{itemize}[noitemsep,topsep=2pt]
\item $v_\mathbf{x}(\phi,w)\leq v_{\mathbf{x}'}(\phi',w')$ for any $w:\mathbf{x}:\phi\leqslant w':\mathbf{x}':\phi'\in\mathcal{B}$ with $\mathbf{x},\mathbf{x}'\in\{1,2\}$.
\item $v_\mathbf{x}(\phi,w)\leq c$ for any $w:\mathbf{x}:\phi\leqslant{c}\in\mathcal{B}$ with ${c}\in\{0,1\}$.
\end{itemize}
\end{definition}
\begin{theorem}[Completeness]\label{theorem:Gconstraintcompleteness}
$\phi$ is $\fbKGsquare$ valid  iff it has a $\mathcal{T}(\fbKGsquare)$ proof.
\end{theorem}
\begin{proof}
We consider only the $\fbKGsquare$ case since $\fbKbiG$ can be handled the same way. For soundness, we check that if the premise of the rule is realised, then so is at least one of its conclusions. We consider the cases of $\rightarrow_1\!\leqslant$ and $\Box_1\!\lesssim$. Assume that $w\!:\!1\!:\phi\rightarrow\phi'\!\leqslant\!\mathfrak{X}$ is realised and assume w.l.o.g.\ that $\mathfrak{X}=u\!:\!2\!:\!\psi$. It is clear that either $v_2(\psi,u)=1$ or $v_2(\psi,u)<1$. In the first case, $\mathfrak{X}\geqslant1$ is realised. In the second case, we have that $v_1(\phi,w)>v_1(\phi',w)$ and $v_1(\phi',w)\leqslant v_2(\psi,u)$. Thus, $\mathfrak{X}<1$, $w\!:\!1\!:\!\phi>w\!:\!1\!:\!\phi'$, and $w\!:\!1\!:\!\phi'\leqslant u\!:\!1\!:\!\psi$ are realised as well, as required.

For $\Box_1\!\lesssim$, assume that $w\!:\!1\!:\Box\phi\!\leqslant\!\mathfrak{X}$ is realised and assume w.l.o.g.\ that $\mathfrak{X}=u\!:\!2\!:\!\psi$. Thus, $v_1(\Box\phi,w)\leqslant v_2(\psi,u)$ Then, since the model is finitely branching, there is an accessible state $w''$ s.t.\ $v_1(\phi,w)\leqslant v_2(\psi,u)$. Thus, $w''\!:\!1\!:\phi\!\leqslant\!\mathfrak{X}$ is realised too.

As no closed branch is realisable, the result follows.

For completeness, we show that every complete open branch $\mathcal{B}$ is realisable. We construct the model as follows. We let $W=\{w:w\text{ occurs in }\mathcal{B}\}$, and set $R=\{\langle w,w'\rangle:w\mathsf{R}w'\in\mathcal{B}\}$. Now, it remains to construct the suitable valuations.

For $i\in\{1,2\}$, if $w\!:\!i\!:\!p\geqslant1\in\mathcal{B}$, we set $v_i(p,w)=1$. If $w\!:\!i\!:\!p\leqslant0\in\mathcal{B}$, we set $v_i(p,w)=0$. To set the values of the remaining variables $q_1$, \ldots, $q_n$, we proceed as follows. Denote $\mathcal{B}^+$ the transitive closure of $\mathcal{B}$ under $\lesssim$ and let
\[[w\!:\!\mathbf{x}\!:\!q_i]\!=\!\left\{w'\!:\!\mathbf{x}'\!:\!q_j \; \left| \; \begin{matrix}w\!:\!\mathbf{x}\!:\!q_i\leqslant w'\!:\!\mathbf{x}'\!:\!q_j\in\mathcal{B}^+\text{ and }w\!:\!\mathbf{x}\!:\!q_i\!<\!w'\!:\!\mathbf{x}'\!:\!q_j
\notin\mathcal{B}^+\\
\text{or}\\
w\!:\!\mathbf{x}\!:\!q_i\geqslant w'\!:\!\mathbf{x}'\!:\!q_j\in\mathcal{B}^+\text{ and }w\!:\!\mathbf{x}\!:\!q_i\!>\!w'\!:\!\mathbf{x}'\!:\!q_j\notin\mathcal{B}^+
\end{matrix}\right.\right\}\]
It is clear that there are at most $2\cdot n\cdot|W|$ $[w\!:\!\mathbf{x}\!:\!q_i]$'s  since the only possible loop in $\mathcal{B}^+$ is $w_{i_1}\!:\!\mathbf{x}\!:\!r\leqslant\ldots\leqslant w_{i_1}\!:\!\mathbf{x}\!:\!r$, but in such a loop all elements belong to $[w_{i_1}\!:\!\mathbf{x}\!:\!r]$. We put $[w\!:\!\mathbf{x}\!:\!q_i]\prec[w'\!:\!\mathbf{x}'\!:\!q_j]$ iff there are $w_k\!:\!\mathbf{x}\!:\!r\in[w\!:\!\mathbf{x}\!:\!q_i]$ and $w'_k\!:\!\mathbf{x}'\!:\!r'\in[w'\!:\!\mathbf{x}'\!:\!q_j]$ s.t. $w_k\!:\!\mathbf{x}\!:\!r<w'_k\!:\!\mathbf{x}'\!:\!r'\in\mathcal{B}^+$.

We now set the valuation of these variables as follows
\begin{align*}
v_\mathbf{x}(q_i,w)=\dfrac{|\{[w'\!:\!\mathbf{x}'\!:\!q']\mid[w'\!:\!\mathbf{x}'\!:\!q']\prec[w\!:\!\mathbf{x}\!:\!q_i]\}|}{2\cdot n\cdot|W|}
\end{align*}
Note that if some $\phi$ contains $s$ but $\mathcal{B}^+$ contains no inequality with it, the above definition ensures that $s$ is going to be evaluated at $0$. Thus, all constraints containing only variables are satisfied.

It remains to show that all other constraints are satisfied. For that, we prove that if at least one conclusion of the rule is satisfied, then so is the premise. The propositional cases are straightforward and can be tackled in the same manner as in~\cite[Theorem~2]{BilkovaFrittellaKozhemiachenko2021}. We consider only the case of $\lozenge_2\!\gtrsim$. Assume w.l.o.g. that $\gtrsim=\geqslant$ and $\mathfrak{X}=u\!:\!1\!:\!\psi$. Since $\mathcal{B}$ is complete, if $w\!:2\!:\!\lozenge\phi\geqslant u\!:\!1\!:\!\psi\in\mathcal{B}$, then for any $w'$ s.t. $w\mathsf{R}w'\in\mathcal{B}$, we have $w'\!:2\!:\!\phi\geqslant u\!:\!1\!:\!\psi\in\mathcal{B}$, and all of them are realised by $\mathfrak{M}$. But then $w\!:2\!:\!\lozenge\phi\geqslant u\!:\!1\!:\!\psi$ is realised too, as required.
\end{proof}

\begin{theorem}\label{theorem:complexity}
\begin{enumerate}
\item[]
\item Let $\phi\in\bimodalLsquare$ be \emph{not $\fbKGsquare$ valid}, and let $|\phi|$ denote the number of symbols in it. Then there is a model $\mathfrak{M}$ of the size $O(|\phi|^{|\phi|})$ and depth $O(|\phi|)$ and $w\in\mathfrak{M}$ s.t. $v_1(\phi,w)\neq1$.
\item $\fbKGsquare$ validity and satisfiability\footnote{Satisfiability and falsifiability (non-validity) are reducible to each other using $\coimplies$: $\phi$ is satisfiable iff ${\sim\sim}(\phi\coimplies\mathbf{0})$ is falsifiable; $\phi$ is falsifiable iff ${\sim\sim}(\mathbf{1}\coimplies\phi)$ is satisfiable.
} are $\mathsf{PSPACE}$-complete.
\end{enumerate}
\end{theorem}
\begin{proof}
We begin with 1. By theorem~\ref{theorem:Gconstraintcompleteness}, if $\phi$ is \emph{not $\fbKGsquare$ valid}, we can build a~falsifying model using tableaux. It is also clear from the rules on fig.~\ref{fig:modalrules} that the depth of the constructed model is bounded from above by the maximal number of nested modalities in $\phi$. The width of the model is bounded by the maximal number of modalities on the same level of nesting. The sharpness of the bound is obtained using the embedding of $\mathbf{K}$ into $\fbKGsquare$ since $\mathbf{K}$ is complete w.r.t.\ finitely branching models and it is possible to force shallow trees of exponential size in $\mathbf{K}$ (cf., e.g.~\cite[\S6.7]{BlackburndeRijkeVenema2010}). The embedding also entails $\mathsf{PSPACE}$-hardness. It remains to tackle membership.

First, observe from the proof of theorem~\ref{theorem:Gconstraintcompleteness} that $\phi(p_1,\ldots,p_n)$ is satisfiable (falsifiable) on $\mathfrak{M}=\langle W,R,v_1,v_2\rangle$ iff there are $v_1$ and $v_2$ that give variables values from $\mathsf{V}=\left\{0,\frac{1}{2\cdot n\cdot|W|},\ldots,\frac{2\cdot n\cdot|W|-1}{2\cdot n\cdot|W|},1\right\}$ under which $\phi$ is satisfied (falsified).

As we mentioned, $|W|$ is bounded from above by $k^{k+1}$ with $k$ being the number of modalities in $\phi$. Therefore, we replace structural constraints with labelled formulas of the form $w\!:\!i\!:\!\phi\!=\!\mathsf{v}$ ($\mathsf{v}\in\mathsf{V}$) avoiding comparisons of values of formulas in different states. As expected, we close the branch if it contains $w\!:\!i\!:\!\psi\!=\!\mathsf{v}$ and $w\!:\!i\!:\!\psi\!=\!\mathsf{v}'$ for $\mathsf{v}\neq\mathsf{v}'$.

Now we replace the rules with the new ones that work with labelled formulas instead of structural constraints. Below, we give as an example new rules for $\rightarrow$ and $\lozenge$\footnote{Intuitively, for a value $1>\mathsf{v}>0$ of $\lozenge\phi$ at $w$, we add a new state that witnesses $\mathsf{v}$, and for a~state on the branch, we guess a~value smaller than $\mathsf{v}$. Other modal rules can be rewritten similarly.} (with $|\mathsf{V}|=m+1$):
\[\dfrac{w\!:\!1\!:\!\phi\rightarrow\phi'\!=\!1}{w\!:\!1\!:\!\phi=0\left|\begin{matrix}w\!:\!1\!:\!\phi\!=\!\frac{1}{m+1}\\w\!:\!1\!:\!\phi'\!=\!\frac{1}{m+1}\end{matrix}\right.\left|\begin{matrix}w\!:\!1\!:\!\phi\!=\!\frac{1}{m+1}\\w\!:\!1\!:\!\phi'\!=\!\frac{2}{m+1}\end{matrix}\right|\ldots\left|\begin{matrix}w\!:\!1\!:\!\phi\!=\!\frac{m-1}{m+1}\\w\!:\!1\!:\!\phi'\!=\!\frac{m}{m+1}\end{matrix}\right|w\!:\!1\!:\!\phi'\!=\!1}\]
\begin{align*}
\dfrac{w\!:\!1\!:\!\lozenge\phi\!=\!\frac{r}{m+1}}{w\mathsf{R}w'';w''\!:\!1\!:\!\phi\!=\!\frac{r}{m+1}}
&&
\dfrac{w\!:\!1\!:\!\lozenge\phi\!=\!\frac{r}{m+1};w\mathsf{R}w'}{w'\!:\!1\!:\!\phi\!=\!0\mid\ldots\mid w'\!:\!1\!:\!\phi\!=\!\frac{r-1}{m+1}}
\end{align*}

We now show how to build a satisfying model for $\phi$ using polynomial space. We begin with $w_0\!:\!1\!:\phi\!=\!1$ and start applying propositional rules (first, those that do not require branching). If we implement a branching rule, we pick one branch and work only with it: either until the branch is closed, in which case we pick another one; until no more rules are applicable (then, the model is constructed); or until we need to apply a modal rule to proceed. At this stage, we need to store only the subformulas of $\phi$ with labels denoting their value at~$w_0$.

Now we guess a~modal formula (say, $w_0\!:\!2\!:\!\Box\chi\!=\!\frac{1}{m+1}$) whose decomposition requires an introduction of a~new state ($w_1$) and apply this rule. Then we apply all modal rules that use $w_0\mathsf{R}w_1$ as a premise (again, if those require branching, we guess only one branch) and start from the beginning with the propositional rules. If we reach a contradiction, the branch is closed. Again, the only new entries to store are subformulas of $\phi$ (now, with fewer modalities), their values at $w_1$, and a~relational term $w_0\mathsf{R}w_1$. Since the depth of the model is $O(|\phi|)$ and since we work with modal formulas one by one
, we need to store subformulas of $\phi$ with their values $O(|\phi|)$ times, so, we need only $O(|\phi|^2)$ space.

Finally, if no rule is applicable and there is no contradiction, we mark $w_0\!:\!2\!:\!\Box\chi\!=\!\frac{1}{m+1}$ as ‘safe’. Now we \emph{delete all entries of the tableau below it} and pick another unmarked modal formula that requires an introduction of a new state. Dealing with these one by one allows us to construct the model branch by branch. But since the length of each branch of the model is bounded by $O(|\phi|)$ and since we delete \emph{branches of the model} once they are shown to contain no contradictions, we need only polynomial space.
\end{proof}

We end the section with two simple observations. First, theorems~\ref{theorem:Gconstraintcompleteness} and~\ref{theorem:complexity} are applicable both to $\fbKbiG$ and $\fbKGsquare$ because the latter is conservative over the former. Secondly, since $\KGsquare$ and $\KbiG$ are conservative over $\KG^c$ and since $\mathbf{K}$ can be embedded in $\KG^c$, the lower bounds on complexity of a classical modal logic of some class of frames $\mathbb{K}$ and $\Gsquare$ modal logic of $\mathbb{K}$ will coincide.
\section{Concluding remarks}\label{sec:conclusion}
In this paper, we developed a crisp modal expansion of the two-dimensional G\"{o}del logic $\Gsquare$ as well as an expansion of bi-G\"{o}del logic with $\Box$ and $\lozenge$ both for crisp and fuzzy frames. We also established their connections with modal G\"{o}del logics, and gave a complexity analysis of their finitely branching fragments.

The following steps are: to study the proof theory of $\KGsquare$ and $\fbKGsquare$: both in the form of Hilbert-style and sequent calculi; establish the decidability (or lack thereof) for the case of $\KGsquare$. Moreover, two-dimensional treatment of information invites for different modalities, e.g. those formalising aggregation strategies given in~\cite{BilkovaFrittellaMajerNazari2020} --- in particular, the cautious one (where the agent takes minima / infima of \emph{both} positive and negative supports of a given statement) and the confident one (whereby the maxima / suprema are taken). Last but not least, while in this paper we assumed that our \emph{access} to sources is crisp, one can argue that the \emph{degree of our bias} towards the given source can be formalised via \emph{fuzzy} frames. Thus, it would be instructive to construct a fuzzy version of $\KGsquare$.

In a broader perspective, we plan to provide a general treatment of two-dimensional modal logics of uncertainty. Indeed, within our project~\cite{BilkovaFrittellaMajerNazari2020,BilkovaFrittellaKozhemiachenko2021}, we are formalising reasoning with heterogeneous and possibly incomplete and inconsistent information (such as crisp or fuzzy data, personal beliefs, etc.) in a modular fashion. This modularity is required because different contexts should be treated with different logics --- indeed, not only the information itself can be of various nature but the reasoning strategies of different agents even applied to the same data are not necessarily the same either. Thus, since we wish to account for this diversity, we should be able to combine different logics in our approach.
\bibliographystyle{splncs04}

\end{document}